\documentclass[11pt]{article}
\input{packages.sty}
\input{macros.sty}

\begin{document}
\title{An a posteriori error analysis based on non-intrusive spectral projections for systems of random conservation laws}
\author{Jan Giesselmann\thanks{Department of Mathematics, TU Darmstadt, Dolivostra\ss e 15,  64293 Darmstadt, Germany. 
} 
\and Fabian Meyer\thanks{Institute of Applied Analysis and Numerical Simulation, University of Stuttgart, 
Pfaffenwaldring 57, 70569 Stuttgart, Germany. 
\newline F.M., C.R. thank the Baden-W{\"u}rttemberg Stiftung for support via the project 'SEAL'. J.G. thanks the German Research Foundation (DFG) for support of the project via DFG grant GI1131/1-1. $^*$fabian.meyer@mathematik.uni-stuttgart.de}~$^{,*}$
\and Christian Rohde\samethanks }
\date{April 30, 2019}
\providecommand{\keywords}[1]{{\textit{Key words:}} #1\\ \\}
\providecommand{\class}[1]{{\textit{AMS subject classifications:}} #1}
\maketitle

\begin{abstract}  \noindent
We present an a posteriori error analysis for one-dimensional random
hyperbolic systems of  conservation laws. 
For the discretization of the random space we consider the Non-Intrusive Spectral Projection method, the spatio-temporal discretization 
uses the Runge--Kutta Discontinuous Galerkin method.
We derive an a posteriori error estimator
using smooth reconstructions of the numerical solution, which combined with the relative entropy stability framework yields computable error bounds for the space-stochastic discretization error.
Moreover, we show that the estimator admits a splitting into a stochastic and
deterministic part.

\end{abstract}

\section{Introduction}
In this contribution we study numerical schemes for spatially one-dimensional systems of random hyperbolic conservation laws, where the uncertainty
stems from random initial data.
The random space is discretized using the Non-Intrusive Spectral Projection (NISP) method which is based on 
discrete orthogonal projections, cf. \cite{LeMaitreOlivier2002}. 
The resulting deterministic equations are 
discretized by a Runge--Kutta Discontinuous Galerkin (RKDG) method \cite{CockburnShu2001}.
We reconstruct the numerical solutions based on reconstructions for determinstic problems suggested in \cite{GiesselmannDedner16}, see also \cite{GiesselmannMeyerRohde19IMAJNA} for their use in Stochastic Galerkin schemes.
Based on these reconstructions and using the relative entropy framework, cf. \cite[Section 5.2]{Dafermos2016}, we 
derive an a posteriori error bound for the difference between the exact solution of the random
hyperbolic conservation law and its
numerical approximation. We show that the corresponding residual admits a decomposition into three parts:
A spatial part, a stochastic part, and a part which quantifies the quadrature error introduced
by the discrete orthogonal projection.
This decomposition paves the way for novel residual-based adaptive numerical schemes. 

The article is structured as follows: In \secref{sec:prelim} we describe the problem of interest. 
In \secref{sec:discretization} 
the NISP and RKDG method is reviewed and we show how to obtain the reconstruction from our numerical solution.
\secref{sec:aposteriori} presents our main a posteriori error estimate with
decomposition of the residual.
\section{Statement of the Problem} \label{sec:prelim}
Let $(\Omega, \F,\P)$ be a probability space, where $\Omega$ is the set of all elementary events $\omega \in \Omega$, $\F$ is a $\sigma$-algebra
on $\Omega$ and $\P$ is a probability measure.
We consider uncertainties parametrized by a random variable
 $\xi : \Omega\to  \Xi \subset \R$ with probability density function $w_{\xi} :\Xi \to \R_+$.
The random variable induces a probability measure $\tilde{\P}(B):= \P( \xi^{-1}(B))$ for all $B \in \mathcal{B}(\Xi)$ on the measurable space $(\Xi, \mathcal{B}(\Xi))$,
where $\mathcal{B}(\Xi)$ is the corresponding Borel $\sigma$-algebra.
This measure is called the law of $\xi$ and in the following we work on the 
image probability space $(\Xi, \mathcal{B}(\Xi), \tilde{\P})$.
For a second measurable space ($E, \mathcal{B}(E))$, we consider the weighted $L_\xi^p$-spaces equipped with the norm
\begin{align*}
\|f\|_{L_\xi^p(\Xi;E)} :=
\begin{cases}
\Big(\int \limits_\Xi\|f(y)\|_E^p ~w_\xi(y)\mathrm{d}y\Big)^{1/p} = \E\Big(\|f\|_E^p\Big)^{1/p},\quad& 1\leq p<\infty \\
\operatorname{ess~sup}\limits_{y \in \Xi} \|f(y)\|_E,  &p= \infty .
\end{cases} 
\end{align*}
Our problem of interest is the following initial value problem for an one dimensional system of $m \in \N$ random conservation laws, i.e.,
\begin{equation} \label{eq:RIVP}   \tag{RIVP}
\begin{cases}
\partial_t \stochSol(t,x,y) + \partial_x  \Flux(\stochSol(t,x,y))= 0, &(t,x,y)  \in   (0,T) \times \R \times ~\Xi,
\\ \stochSol(0,x,y)= \stochSol^0(x,y),  &(x,y) \in  \R \times ~ \Xi.\notag
\end{cases}
\end{equation} 
Here, $\stochSol(t,x,y) \in \U \subset \R^m$ is the vector of conserved unknown quantities, $\Flux\in C^2(\U;\R^m)$,
is the flux function, $\stochSol^0$ is the uncertain initial condition,
$\mathcal{U}\subset \R^m$ is the state space, which is assumed to be an open set and  $T \in (0,\infty)$ describes the final time.
We assume that \eqref{eq:RIVP} is strictly hyperbolic, i.e.~its Jacobian $\D \Flux(\stochSol)$ has $m$
distinct real eigenvalues.

We say that $(\eta,q)\in C^2(\U;\R)$ forms an entropy/entropy-flux pair
if $\eta$  is strictly convex and if $\eta$ and $q$ satisfy $\D \eta \D \Flux =\D q$.
We assume that the random conservation law \eqref{eq:RIVP} is equipped with at least one
entropy/entropy-flux pair.
Following the definition in \cite{RisebroSchwab2016} for scalar problems,
we call $\stochSol \in \leb{1}{\Xi}{L^1((0,T)\times \R;\U)}$  a random entropy solution of \eqref{eq:RIVP}, if $\stochSol(\cdot,\cdot,y)$
is a classical entropy solution, cf. \cite[Def.~4.5.1]{Dafermos2016}, $\Pas$.
The well-posedness of \eqref{eq:RIVP}, will not be discussed in this article but can
found in \cite{GiesselmannMeyerRohde19}, where existence and uniqueness of random entropy solutions
for \eqref{eq:RIVP} with random flux functions and random initial data with sufficiently small total variation is proven, based on the results of \cite{BressanLeFloch1997}.

\section{Space-Time Stochastic Discretization and Reconstructions} \label{sec:discretization}
For the stochastic discretization of \eqref{eq:RIVP} we use the NISP method, \cite{LeMaitreOlivier2002},
which is based on the (generalized) polynomial chaos expansion which was introduced in \cite{XiuKarniadakis2002}. 
Under the assumption that $\stochSol$ is square-integrable with respect to $\Xi$, 
we expand the solution of \eqref{eq:RIVP} into a generalized Fourier series 
using  a suitable orthonormal basis.

Let ${\{\Psi_{i}(\cdot)\}_{i\in \N}: \Xi \to \R}$ be a  $\Leb_\xi^2(\Xi)$-orthonormal basis,
i.e. for $i,j\in \N$ we have
\begin{equation}
\begin{aligned}\label{eq.orthogonal}
\Big\langle \Psi_{i},\Psi_{j} \Big\rangle := \E\Big(\Psi_{i} \Psi_{j}\Big) 
= \int \limits_{\randomSpace} \Psi_{i}(y)\Psi_{j}(y) w_{\xi}(y) ~\mathrm{d}y=\delta_{ij}.
\end{aligned}
\end{equation}
Following \cite{XiuKarniadakis2002}, the random entropy solution $\stochSol$ can be written as
\begin{align} \label{eq:series}
\stochSol(t,x,y)= \sum \limits_{\globalIdx=0}^\infty \stochSol_\globalIdx(t,x) \Psi_{\globalIdx}(y),
\end{align}
with (deterministic) Fourier modes $u_\globalIdx=u_\globalIdx(t,x)$ satisfying
\begin{align}\label{eq:modes}
\stochSol_\globalIdx(t,x)=\E\Big(\stochSol(t,x,\cdot) \Psi_\globalIdx(\cdot)\Big) \qquad\forall \globalIdx\in \N.
\end{align}
The NISP method approximates the modes in \eqref{eq:modes} via a discrete orthogonal projection, i.e., numerical quadrature. We denote $(\quadOrder+1) \in \N$ quadrature points and weights   by
$\{y_{\globalNispIdx}\}_{\globalNispIdx=0}^{\quadOrder}$, $\{w_{\globalNispIdx}\}_{\globalNispIdx=0}^{\quadOrder}$,  
and approximate
\begin{align}\label{eq:NISPMode}
\stochSol_\globalIdx(t,x)= \int \limits_{\randomSpace} \stochSol(t,x,y) \Psi_\globalIdx(y) w_{\xi}(y) ~\mathrm{d}y \approx
\sum \limits_{\globalNispIdx=0}^R \stochSol(t,x,y_\globalNispIdx) \Psi_\globalIdx(y_\globalNispIdx)w_\globalNispIdx  =: \hat{u}_i
\quad \text{ for }\globalIdx \in \N.
\end{align}
In a second step the NISP method truncates \eqref{eq:series} after the $M$-th mode, i.e.,
\begin{align}
\stochSol(t,x,y) \approx \sum \limits_{\globalIdx=0}^\polyOrder \hat{u}_\globalIdx(t,x) \Psi_{\globalIdx}(y).
\end{align} 
For any $l=0,\ldots,\quadOrder$, the random entropy solution $u$ of \eqref{eq:RIVP} evaluated at quadrature point $\{y_\globalNispIdx\}_{\globalNispIdx=0}^{\quadOrder}$, is denoted by $\stochSol(\cdot,\cdot,y_\globalNispIdx)$ and it
is an entropy solution of the deterministic version of \eqref{eq:RIVP}, i.e.  of
\newtagform{mytag}{(}{$)_\globalNispIdx$}
\usetagform{mytag}
\begin{equation} \label{eq:DIVP}   \tag{DIVP}
\begin{cases}
\partial_t \stochSol(t,x,y_\globalNispIdx) + \partial_x  \Flux(\stochSol(t,x,y_\globalNispIdx))= 0, &(t,x)  \in   (0,T) \times \R,
 \\ \stochSol(0,x,y_\globalNispIdx)= \stochSol^0(x,y_\globalNispIdx),  & x \in  \R\notag.
\end{cases}
\end{equation} 
The deterministic hyperbolic systems \eqref{eq:DIVP} can be discretized\newtagform{mytag2}{(}{)}
\usetagform{mytag2}by a suitable numerical method. We use the RKDG
method as described in \cite{CockburnShu2001}. We denote the corresponding numerical solution 
of \usetagform{mytag} \eqref{eq:DIVP} \usetagform{mytag2}at quadrature point $\{y_\globalNispIdx\}_{\globalNispIdx=0}^{\quadOrder}$ 
and at points $\{t_n(y_\globalNispIdx)\}_{n=0}^{\ntCells(y_\globalNispIdx)}$, $\ntCells(y_\globalNispIdx) \in \N$, 
in time by $\numSol^n(\cdot,y_\globalNispIdx) \in \DG{p}$,
where 
$$\DG{p}:= \{v: \R \to \R^m ~|~v\mid_{K} \in \P_p(K;\R^m),~K \in \T \},  $$
is the corresponding DG space of polynomials of degree $p \in \N$, 
associated with a uniform triangulation $\T$ of $\R$. 
Let us assume that the time partition $\{t_n\}_{n=0}^{\ntCells}$ and the triangulation $\T$ used for \usetagform{mytag} \eqref{eq:DIVP} \usetagform{mytag2}
are the same for every quadrature point  $\{y_\globalNispIdx\}_{\globalNispIdx=0}^{\quadOrder}$.
The numerical approximation of \eqref{eq:RIVP} at time $t=t_n$ can then be written as
\begin{align} \label{def:numSolNISP}
\numSol^n(x,y):=  \sum \limits_{\globalIdx=0}^{\polyOrder}
 \Big( \sum \limits_{\globalNispIdx=0}^{\quadOrder} \numSol^n(x,y_{\globalNispIdx})
  \Psi_\globalIdx(y_\globalNispIdx) w_\globalNispIdx \Big) \Psi_\globalIdx(y). 
\end{align} 
The proof of the a posterior error estimate in \thmref{thm:stochApostNumSol} uses the relative entropy framework, cf. \cite[Section 5.2]{Dafermos2016},
which requires one quantity which is at least Lipschitz continuous in space and time. To this end we reconstruct the
numerical solution so that we obtain a Lipschitz continuous function.
To avoid technical overhead, we do not elaborate upon this process here, but refer to \cite{GiesselmannDedner16,GiesselmannMeyerRohde19IMAJNA},
where a detailed description can be found.

The reconstruction provides us with a computable space-time reconstruction $\reconst{st}(y_{\globalNispIdx}) \in W_\infty^1((0,T);V_{p+1}^s\cap C^0(\R))$  of the 
numerical solution $\{\numSol^n(y_\globalNispIdx)\}_{n=0}^{\ntCells} \subset \DG{p}$, for each quadrature point $\{y_\globalNispIdx\}_{\globalNispIdx=0}^{\quadOrder}$.
This allows us to define a space-time residual as follows.
\begin{definition}[Space-time residual]
For all  $\globalNispIdx=0,\ldots, \quadOrder$, we define $\residual{st}(y_\globalNispIdx)\in \Leb^2((0,T)\times \R;\R^m)$ by
\begin{align}\label{def:STResidualColl}
\residual{st}(y_\globalNispIdx):= \partial_t \reconst{st}(y_\globalNispIdx)+ \partial_x \Flux(\reconst{st}(y_\globalNispIdx))
\end{align}
to be the space-time residual associated with the quadrature point $y_\globalNispIdx$.
\end{definition} 
Next we define the reconstructed mode, the
space-time-stochastic reconstruction and the space-time-stochastic residual. The latter is obtained by plugging  the space-time-stochastic reconstruction into the random conservation law \eqref{eq:RIVP}.
\begin{definition}[Space-time-stochastic reconstruction and residual] 
Let $\{\reconst{st}(y_\globalNispIdx) \}_{\globalNispIdx=0}^R: (0,T)\times \R \to \R^m$  
be the sequence of space-time reconstructions at quadrature points $\{y_\globalNispIdx\}_{\globalNispIdx=0}^{\quadOrder}$.
The reconstructed modes of \eqref{eq:NISPMode} are defined as
\begin{align}\label{def:STModeNISP}
\reconst{st}_{\globalIdx}:= \sum \limits_{\globalNispIdx=0}^R \reconst{st}(y_{\globalNispIdx}) \Psi_{\globalIdx}(y_\globalNispIdx) w_{\globalNispIdx},
\end{align}
for $\globalIdx =0,\ldots, \polyOrder$. The space-time-stochastic reconstruction $\reconst{sts}: (0,T)\times \R \times ~\Xi \to \R^m$
is defined as
\begin{align}\label{def:STSReconstNISP}
\reconst{sts}(t,x,y):= \sum \limits_{\globalIdx =0}^\polyOrder \reconst{st}_{\globalIdx}(t,x) \Psi_\globalIdx(y).
\end{align}
Finally, we define the space-time-stochastic residual $\residual{sts}\in \leb{2}{\Xi}{\Leb^2((0,T)\times \R;\R^m)}$ by
\begin{align}\label{eq:STSResidual}
\residual{sts}:= \deriv{t} \reconst{sts}+ \deriv{x}\Flux(\reconst{sts}).
\end{align}
\end{definition}
This residual is crucial in the upcoming error analysis.

\section{A Posteriori Error Estimate and Error Indicators}\label{sec:aposteriori}
Before stating the main a posterior error estimate, let us note that derivatives of the flux function and the entropy are bounded on any compcat subset $\mathcal{C}$ of the state space.
These bounds enter the upper bound in Theorem \ref{thm:stochApostNumSol}.
Let $\C \subset \U$ be convex and compact. Due to $\Flux \in C^2(\U,\R^m)$ and $\eta \in C^2(\U,\R^m)$ strictly convex 
there  exist constants $0<\fluxConstant < \infty$ and $0<\lowerEta < \upperEta< \infty$, s.t.
\begin{align*}
|v^\top H \Flux(u)v|\leq \fluxConstant|v|^2, \quad \lowerEta|v|^2 \leq v^\top H \eta(u)v \leq \upperEta |v|^2, \quad \forall v\in \R^m, \forall u\in \C.
\end{align*}
Here $H \Flux$ denotes the Hessian (i.e. the tensor of second order derivatives) of the flux function and $H\eta$ the Hessian of the entropy $\eta$.
We now have all ingredients together to state the following a posteriori error estimate
that can be directly derived from \cite{GiesselmannMakridakisPryer15}.
\begin{theorem}[A posteriori error bound for the numerical solution] 
\label{thm:stochApostNumSol}
Let $u$ be the random entropy solution of \eqref{eq:RIVP}. Then, for any  $n=0,\ldots,\ntCells$, the difference between $u(t_n, \cdot,\cdot)$ and the numerical solution $\numSol^n$ from  \eqref{def:numSolNISP} satisfies
\begin{align*}
\|\stochSol(t_n,\cdot,\cdot)-\numSol^n(\cdot,\cdot)& \|_{\leb{2}{{\Xi} }{\Leb^2(\R)}}^2 \\
  & \leq  ~ 2\|\reconst{sts}(t_n,\cdot,\cdot)- \numSol^n(\cdot,\cdot)\|_{\leb{2}{{\Xi} }{\Leb^2(\R)}}^2 
  \\ 
 & + 2 \lowerEta^{-1} \Big( \stsResidual(t_n) + \upperEta \stsResidual_0 \Big) 
 \\& \times \exp\Big( \lowerEta^{-1}\int \limits_0^{t_n} \Big(\upperEta \fluxConstant \|\partial_x \reconst{sts}(t,\cdot,\cdot)\|_{\leb{\infty}{\Xi}{\Leb^\infty(\R)}} + \upperEta^2\Big)~\mathrm{d}t\Big),
\end{align*}
with 
\begin{align*}
\stsResidual(t_n)&:= \|\residual{sts}(\cdot,\cdot,\cdot)\|_{\leb{2}{\Xi}{\Leb^2((0,t_n)\times \R)}}^2, \\ \stsResidual_0&:= \|u^0(\cdot,\cdot)-\reconst{sts}(0,\cdot,\cdot)\|_{\leb{2}{\Xi}{\Leb^2( \R)}}^2.
\end{align*}
\end{theorem}
\begin{proof}
We apply \cite[Lemma 5.1]{GiesselmannMakridakisPryer15} path-wise in $\Xi$,
integrate over $\Xi$ and use Gronwall's inequality to bound
$\|\stochSol(t_n,\cdot,\cdot)- \reconst{sts}(t_n,\cdot,\cdot)\|_{\leb{2}{{\Xi} }{\Leb^2(\R)}}$
by the second term in the inequality.
The final estimate then follows using the triangle inequality.
\end{proof}
In \thmref{thm:stochApostNumSol} the error between the numerical solution and the entropy solution is bounded
by  the error in the initial condition, the difference between the numerical solution and its reconstruction
and the contribution of the space-time stochastic residual $\residual{sts}$ from \eqref{eq:STSResidual}, quantified by $\stsResidual$.
We would like to distinguish between errors that arise from discretizing the random space and from discretizing the physical space.
Therefore, we show in \lemref{eq:splittingNISP} a splitting of the space-time-stochastic residual $\residual{sts}$ into three parts.~Namely a deterministic residual, which corresponds to the spatial error  
when approximating \usetagform{mytag} \eqref{eq:DIVP} \usetagform{mytag2} using the RKDG method, a quadrature residual that reflects the quadrature error
from the discrete orthogonal projection in \eqref{eq:NISPMode} and
a stochastic cut-off error, which occurs when truncating the infinite Fourier series in \eqref{eq:series}.

\begin{lemma}[Orthogonal decomposition of the space-time-stochastic residual] \label{lemma:splitting}
The space-time-stochastic residual $\residual{sts}$ from \eqref{eq:STSResidual} admits 
the following orthogonal decomposition,
\begin{align}\label{eq:splittingNISP}
\residual{sts}= \sum \limits_{j=0}^M  \Big(\residual{det}_j + \residual{sq}_j\Big) \Psi_j 
+ \sum \limits_{j>M}^\infty \residual{sc}_j   \Psi_j,
\end{align}
where 
\begin{align*}
& \residual{det}_{j} := \sum \limits_{\globalNispIdx=0}^{R} \fatresidual{st}(y_\globalNispIdx) \Psi_j(y_\globalNispIdx) w_\globalNispIdx \quad \text{ for }j= 0, \ldots, M  \\
& \residual{sq}_{j} :=  \Big \langle \deriv{x} \Flux\Big(\sum \limits_{\globalIdx=0}^\polyOrder \fatreconst{st}(y_\globalIdx) \Psi_\globalIdx\Big), \Psi_j\Big \rangle
 - \sum \limits_{\globalNispIdx=0}^{R} \deriv{x}\Flux(\reconst{st}(y_\globalNispIdx)) \Psi_j(y_\globalNispIdx) w_\globalNispIdx\quad \text{ for }j= 0, \ldots, \polyOrder \\
 & \residual{sc}_{j}:= \Big \langle \deriv{x} \Flux\Big(\sum \limits_{\globalIdx=0}^\polyOrder \fatreconst{st}(y_\globalIdx) \Psi_\globalIdx \Big), \Psi_j \Big \rangle \quad\text{ for } j>M
\end{align*}
are called the $j$-th mode of the deterministic, stochastic quadrature and stochastic cut-off residual.
Moreover, we have 
\begin{align} \label{eq:splittingNISPNorm}
\stsResidual(t)=\|\residual{sts}\|_{{\leb{2}{\Xi}{\Leb^2((0,t)\times \R}}}^2 &=  \sum \limits_{\globalIdx=0}^M  \|\residual{det}_\globalIdx + \residual{sq}_\globalIdx\|_{\Lp{2}{(0,t)\times \R}}^2 + 
\sum \limits_{\globalIdx>M}^\infty \| \residual{sc}_\globalIdx \|_{\Lp{2}{(0,t)\times \R}}^2  \nonumber\\
&\leq 2 \detResidual(t) + 2 \quadResidual(t) + \stochResidual(t),
\end{align}
where, for any $t\in(0,T)$,
\begin{align*}
&\detResidual(t):= \sum \limits_{\globalIdx=0}^M  \|\residual{det}_\globalIdx\|_{\Lp{2}{(0,t)\times \R}}^2, \
\quadResidual(t):= \sum \limits_{\globalIdx=0}^M  \|\residual{sq}_\globalIdx\|_{\Lp{2}{(0,t)\times \ \R}}^2,\\
& \stochResidual(t) :=\sum \limits_{\globalIdx>M}^\infty \| \residual{sc}_\globalIdx \|_{\Lp{2}{(0,t)\times \R}}^2.
\end{align*}
\end{lemma}

\begin{proof}
We recall that the space-time reconstruction 
$\reconst{st}(y_\globalNispIdx)$ satisfies
\begin{align} \label{eq:determinsticResidualNISP} 
\residual{st}(y_\globalNispIdx)= \deriv{t} \reconst{st}(y_\globalNispIdx) + \deriv{x} \Flux(\reconst{st}(y_\globalNispIdx) )
\end{align}
for all $\globalNispIdx=0,\ldots,R$.
Moreover, the reconstructed mode $\reconstNISP{j}{st}$ was defined as (cf. \eqref{def:STModeNISP})
\begin{align} \label{eq:stResidualModeNISP}
\reconstNISP{j}{st} = \sum \limits_{\globalNispIdx=0}^R \reconst{st}(y_\globalNispIdx) \Psi_{j}(y_\globalNispIdx) w_\globalNispIdx
\end{align}
for all $j=0,\ldots,M$.
Multiplying \eqref{eq:determinsticResidualNISP} by $\Psi_j (y_\globalNispIdx) w_\globalNispIdx$ and suming over $l=0,\ldots, R$ yields, using \eqref{eq:stResidualModeNISP}, the 
following relationship
\begin{align}\label{eq:relationResidualsNISP}
\sum \limits_{\globalNispIdx=0}^{R} \fatresidual{st}(y_\globalNispIdx) \Psi_j(y_\globalNispIdx) w_\globalNispIdx = \partial_t  \reconstNISP{j}{st} + \sum \limits_{\globalNispIdx=0}^{R} \deriv{x}\Flux(\reconst{st}(y_\globalNispIdx)) \Psi_j(y_\globalNispIdx) w_\globalNispIdx .
\end{align}
By definition of the space-time-stochastic residual we have
\begin{align*}
\residual{sts}= \deriv{t} \reconst{sts} + \deriv{x}\Flux(\reconst{sts})= \deriv{t}\Big(\sum \limits_{\globalIdx =0}^\polyOrder \reconst{st}_{\globalIdx} \Psi_\globalIdx\Big) + \deriv{x} \Flux\Big(\sum \limits_{\globalIdx =0}^\polyOrder \reconst{st}_{\globalIdx} \Psi_\globalIdx \Big).
\end{align*}  
Let us begin by studying the $j$-th mode of $\residual{sts}$ for $j=0,\ldots,M$. In this case the orthogonality relation \eqref{eq.orthogonal} 
yields 
\begin{align}\label{eq:projectionResidual}
\Big \langle \fatresidual{sts}, \Psi_j\Big \rangle  &= \Big \langle \deriv{t} \reconst{sts} + \deriv{x}\Flux(\reconst{sts}), \Psi_j\Big \rangle  = \deriv{t} \reconst{st}_j+ \Big \langle \deriv{x} \Flux\Big(\sum \limits_{\globalIdx = 0}^\polyOrder \reconst{st}_{\globalIdx} \Psi_\globalIdx \Big), \Psi_j \Big \rangle.
\end{align}
Using (\ref{eq:relationResidualsNISP}) we obtain
\begin{align} \label{eq:NiSPSplitting1}
\Big \langle \residual{sts}, \Psi_j \Big \rangle  =& \sum \limits_{\globalNispIdx=0}^{R} \fatresidual{st}(y_\globalNispIdx) \Psi_j(y_\globalNispIdx) w_\globalNispIdx
\\ & + 
 \Big \langle \deriv{x} \Flux\Big(\sum \limits_{\globalIdx = 0}^{\polyOrder} \reconst{st}_{\globalIdx} \Psi_i\Big), \Psi_j \Big \rangle
 - \sum \limits_{\globalNispIdx=0}^{R} \deriv{x}\Flux(\reconst{st}(y_\globalNispIdx)) \Psi_j(y_\globalNispIdx) w_\globalNispIdx
  = \residual{det}_{j}  + \residual{sq}_{j} . \notag
\end{align}
For $j>M$ the $j$-th moment of $\residual{sts}$ is
\begin{align}\label{eq:NiSPSplitting2}
\Big \langle \fatresidual{sts}, \Psi_j  \Big \rangle  = \Big \langle \deriv{x} \Flux\Big(\sum \limits_{\globalIdx=0}^\polyOrder \fatreconst{st}_{\globalIdx} \Psi_\globalIdx\Big), \Psi_j \Big \rangle = \residual{sc}_{j} .
\end{align}
Formula \eqref{eq:splittingNISP} then follows from \eqref{eq:NiSPSplitting1} and \eqref{eq:NiSPSplitting2}. Formula \eqref{eq:splittingNISPNorm} is an application of
the Pythagorean theorem for  $L_\xi^2(\Xi)$.
\end{proof}
Putting together \thmref{thm:stochApostNumSol} and \lemref{lemma:splitting} we obtain our main 
result, the following a posteriori error estimate with separable error bounds.
\begin{theorem}[A posteriori error bound for the numerical solution with error splitting]  \label{thm:apostSplitting}
Let $u$ be the random entropy solution of \eqref{eq:RIVP}. Then, for any $n=0,\ldots,\ntCells$, the difference between $u(t_n,\cdot, \cdot)$ and  $\numSol^n$ from  \eqref{def:numSolNISP} satisfies
\begin{align*}
\|\stochSol(t_n,\cdot,\cdot)-\numSol^n(\cdot,\cdot) &\|_{\leb{2}{{\Xi} }{\Leb^2(\R)}}^2 \\
 & \leq ~  2\|\reconst{sts}(t_n,\cdot,\cdot)- \numSol^n(\cdot,\cdot)\|_{\leb{2}{{\Xi} }{\Leb^2(\R)}}^2 
  \\ 
 &  + 2 \lowerEta^{-1} \Big( 2 \detResidual(t_n) + 2 \quadResidual(t_n) + \stochResidual(t_n) + \upperEta \stsResidual_0 \Big) 
 \\& \times \exp\Big( \lowerEta^{-1}\int \limits_0^{t_n} \Big(\upperEta \fluxConstant \|\partial_x \reconst{sts}(t,\cdot,\cdot)\|_{\leb{\infty}{\Xi}{\Leb^\infty(\R)}} + \upperEta^2\Big)~\mathrm{d}t\Big).
\end{align*}
\end{theorem}

\section{Conclusions and Outlook} 
We derived a novel residual-based a posteriori error bound for
the difference between the entropy solution of \eqref{eq:RIVP} and its
numerical approximation using the NISP method in combination with a RKDG scheme.
Moreover, we proved that the upper bound can be decomposed into three 
parts, where $\detResidual$ quantifies the space-time discretization error of the RKDG scheme,
$\quadResidual$ assesses the quality of the discrete orthogonal projection and
$\stochResidual$ quantifies the stochastic error by truncation of the generalized
polynomial chaos series.
Based on these results, residual-based adaptive numerical schemes, which balance
the contribution of the three different sources of numerical error, can be constructed.
Residual-based space-stochastic adaptive numerical schemes are also considered in \cite{GiesselmannMeyerRohde19}. 

%
\bibliographystyle{siam}
\bibliography{mybib.bib}
\end{document}